\documentclass[reqno]{amsart}
\usepackage[utf8]{inputenc}

\usepackage{array}

\usepackage{
amsmath,
amssymb,
amscd
}

\theoremstyle{plain}
\newtheorem{theorem}{Theorem}
\newtheorem{thm}{Theorem}[section]
\newtheorem{lem}[thm]{Lemma}       
\newtheorem{proposition}[thm]{Proposition}
\newtheorem{corollary}[thm]{Corollary}

\theoremstyle{definition}   
 
\newtheorem{definition}[thm]{Definition}

\newtheorem{rmk}[thm]{Remark}
\numberwithin{equation}{section}

\newcommand{\Cn}{{\mathbb C^n}}
\newcommand{\Bn}{{\mathbb B_n}}

\newcommand{\re}{{\text{Re}}\,}

\newcommand*\diff{\mathop{}\!\mathrm{d}}

\begin{document}

\title[Carleson embedding theorem]{Carleson embedding theorem for an exponential Bergman space on the unit ball}
\author[H. R. Cho, H.-W. Lee, S. Park]{Hong Rae Cho, Han-Wool Lee, and Soohyun Park*}
\address{Hong Rae Cho:
Department of Mathematics, Pusan
National University, Busan 46241, Korea}
\email{chohr@pusan.ac.kr}

\address{Han-Wool Lee:
Department of Mathematics, Pusan
National University, Busan 46241, Korea}
\email{lhw2725@gmail.com}

\address{Soohyun Park:
Department of Mathematics, Pusan
National University, Busan 46241, Korea}
\email{shpark7@pusan.ac.kr}


\thanks{The first author was supported by NRF of Korea (NRF-2020R1F1A1A01048601)}
\thanks{The second author was supported by NRF of Korea (NRF-2020R1I1A1A01074837)}
\thanks{The third author was supported by NRF of Korea (NRF-2021R1I1A1A01049889)}

\begin{abstract}
 We characterize the Carleson measures for an exponential Bergman space on the unit ball of $\Cn$ in terms of the ball induced by the complex Hessian of the logarithm of the weight function. The boundedness (or compactness) of integral operators, Ces\`{a}ro operators and Toeplitz operators, is given using the Carleson measure (or vanishing Carleson measure) characterization.
\end{abstract}

\maketitle

\section{Introduction}\label{sec1}

Let ${\mathbb C^n}$ denote the cartesian product of $n$ copies of the complex field ${\mathbb C}$ for positive integer $n$. For $z=(z_1, \cdots , z_n), w=(w_1, \cdots , w_n) \in {\mathbb C^n}$, the inner product is
$\langle z,w\rangle = \sum_{j=1}^n z_j \overline{w}_j$ 
and the associated norm is 
$\lvert z \rvert^2 = \langle z,z\rangle. $
The open unit ball of ${\mathbb C^n}$ is denoted by $\Bn = \left\{z \in {\mathbb C^n} ; \lvert z \rvert < 1\right\}$ and ${\mathbb D}:= {\mathbb B_1}$.
Let $\diff \mu$ be a Borel measure. For a Borel set $E \subset \Bn$,  
$\mu(E) := \int_{E} \diff \mu.$

The purpose of this article is to demonstrate the Carleson measure characterization for a weighted 
Bergman space with a particular weight $e^{-{\psi}}$ on $\Bn$ with 
$$\psi(z):=\frac{1}{1-\lvert z \rvert^2}.$$
Let $1\leq p < \infty$. Let $\mathcal{O}(\Bn)$ be the space of all holomorphic functions on $\Bn$, 
and let $L^p_{\psi}(\Bn) := L^p(\Bn, e^{-{\psi}} \diff v)$, where $\diff v$ denotes the ordinary Lebesgue measure on ${\mathbb C^n}$.
The exponential Bergman space $A^p_{\psi}(\Bn) := \mathcal{O}(\Bn) \cap L^p_{\psi}(\Bn)$ is the space of holomorphic functions whose $L^p$-norm with the measure $e^{-{\psi}}\diff v$ is bounded, namely, 
\begin{align*}
  \|f\|_{p, \psi}:=\left\{ \int_\Bn \lvert  f(z) \rvert^p e^{-\psi(z)} \diff v(z)  \right\}^{\frac{1}{p}}< \infty.
\end{align*}
The exponential Bergman space $A^p_{\psi}(\Bn)$ is a closed subspace of $L^p_{\psi}(\Bn)$ by Lemma \ref{SMVP}.

When $p=2$, the space 
$A^2_{\psi}(\Bn)$ is a Hilbert space with inner product
\begin{align*}
  \langle f, g\rangle_{\psi} := \int_{\Bn}f(z)\overline{g(z)} e^{-\psi(z)} \diff v(z)
\end{align*}
for $f, g \in A^2_{\psi}(\Bn)$.
Lemma \ref{SMVP} guarantees that each point evaluation $L_zf=f(z)$ is bounded on $A^2_\psi(\Bn)$.  
By the Riesz representation theorem, there is a holomorphic function $K_z \in A^2_{\psi}(\Bn)$ satisfying $f(z) = \langle f, K_z\rangle_{\psi}$. We call ${K_{\psi}(z,w)}:=\overline{K_z(w)}$ the Bergman kernel for $A^2_{\psi}(\Bn)$, i.e., 
\begin{align*}
  f(z) = \int_{\Bn}f(w)K_{\psi}(z,w) e^{-\psi(w)} \diff v(w).
\end{align*}

For a Borel measure $\diff \mu$, if there is a constant $C>0$ satisfying
  \begin{align}\label{def_Carl}
    \int_{\Bn} \lvert  f(z) \rvert^p e^{-\psi(z)} \diff\mu(z) \le C \int_{\Bn} \lvert  f(z) \rvert^p e^{-\psi(z)} \diff v(z),
  \end{align}
then we call the measure $\diff \mu$ is a Carleson measure for $A^p_{\psi}(\Bn)$.
It means that 
the inclusion operator $i_p : A^p_{\psi}(\Bn) \rightarrow L^p(\Bn, e^{- \psi} \diff \mu ) $ is bounded, i.e., $i_p$ is embedding. 

Carleson \cite{MR141789} gave the embedding theorem on the Hardy space for solving the Corona problem on the unit disk $\mathbb D$. 
Results on Carleson measure for the Bergman space on ${\mathbb D}$ was given by Hastings \cite{MR374886}. 
Carleson type measures for exponential type weighted Bergman spaces 
on the unit disk ${\mathbb D}$ was introduced in \cite{oleinik1978embedding},
and has since been studied in numerous exponential type weighted $L^p$ analytic function spaces (see \cite{MR1074530, MR2679024} for ${\mathbb D}$; \cite{MR3472830} for ${\mathbb C}$; 
\cite{MR2891634, MR3010276} for ${\mathbb C^n}$). 

We are going to focus on some results on the unit ball $\Bn$. 
The Carleson measure theorem for the standard weighted and unweighted Bergman spaces on $\Bn$ was proved due to Cima and Wogen \cite{MR650200}. 
Luecking \cite{MR687635} suggested a new method 
which gives simple proofs of the results.
Pau and Zhao \cite{MR3426615} gave 
equivalent conditions of $(p, q)$-Carleson measures for standard weighted Bergman spaces. 
Besides weighted Bergman spaces with the Lebesgue measure, 
Schuster and Varolin considered the M\"{o}bius invariant measure with generalized weights including $(1-\lvert z \rvert^2)^{n+c}$ for some $c>0$ (see Theorem 5.8 of \cite{MR3213550}). 
Now, we consider the exponential Bergman space $A^p_{\psi}(\Bn)$ with $\psi(z)=\frac{1}{1-\lvert z \rvert^2}$, 
which had not been dealt with in any results we mentioned. The weight is rapidly decreasing compared to others.

We introduce the Carleson type embedding theorem for the exponential Bergman space $A^p_{\psi}(\Bn)$. 
First of all, we define the function ${\widehat \mu}_p$ for $p \ge 1$ as 
\begin{align*}
  {\widehat \mu}_p (z) := \frac{1}{\|\Phi_{p,z}\|_{p,\psi}^p} \int_{\Bn} \lvert \Phi_{p,z}(w)\rvert^p e^{-\psi(w)}\diff \mu(w) 
\end{align*}
where $\Phi_{p,z} (w) := e^{\frac{2}{p} \frac{1}{1-\langle w,z\rangle}-\frac{1}{p}\frac{1}{1-\lvert z \rvert^2}}$ is the test function in Lemma \ref{test fn lem}.

\begin{theorem}\label{Mthm} 
  Let $\diff\mu$ be a positive Borel measure. The following statements are equivalent:
  \begin{enumerate}
    \item[(a)] The measure $\diff\mu$ is a Carleson measure for $A^p_{\psi}(\Bn)$;
    \item[(b)] ${\widehat \mu}_p$ is a bounded function on $\Bn$;
    \item[(c)] For $z \in \Bn$ and sufficiently small $r>0$, there is a constant $C>0$ satisfying
    \begin{align*}
      \mu(B_{H}(z,r)) \le C v(B_{H}(z,r)),
    \end{align*}
    where $B_{H}(z,r)$ is the $\psi$-Hessian ball centered at $z$ with radius $r$ (it is defined in Section 2.1).
  \end{enumerate}
\end{theorem}

\noindent For details and the proof of Theorem \ref{Mthm}, see Theorem \ref{Carl measure char}. We also provide the theorem for vanishing Carleson measure in Section \ref{sec3}.

\begin{rmk}
(1) The statement (c) does NOT depend on the number $p$. It shows that if $d\mu$ is a Carleson measure for $A^p_{\psi}(\Bn)$ for some $p$, 
then it also holds for every $p \ge 1$. The fact is analogous to the result on the standard Bergman space.

(2) The implication $(b) \Rightarrow (a)$ means that if inequality \eqref{def_Carl} holds for all test functions, then it also holds for every function in $A^p_{\psi}(\Bn)$.
\end{rmk}

\begin{rmk}
The $\psi$-Hessian ball (cf. \cite{MR2891634, MR3010276, MR3213550, MR1656004}), which is induced by the metric $\left(\frac{\partial^2{\psi}}{\partial z_j\partial\overline{z}_k}\right)_{n \times n}$, 
plays an important role for the proof of Theorem \ref{Mthm}.

(1) With the help of $\psi$-Hessian balls, not only covering lemma (Lemma \ref{covering lem 2}), but also estimates for the test function $\Phi_{p,z}$ (as in the estimate \eqref{test fn_1} 
in Lemma \ref{test fn lem}) 
are obtained.

(2) Lemma \ref{SMVP} is crucial for the proof of Theorem \ref{Mthm}.
It is a weighted sub-mean-value property on the $\psi$-Hessian ball $B_{H}(z,r)$.
\end{rmk}


We should note that using the $\psi$-Hessian ball is suitable for investigating exponential type weighted Bergman spaces on the unit ball rather than using the ball with radius function $\left(\Delta \psi\right)^{-\frac{1}{2}}$. Actually, the ball with $\left(\Delta \psi\right)^{-\frac{1}{2}}$ is helpful tool for studying function spaces with exponential weight on ${\mathbb D}$ \cite{MR2538941, MR2679024} and ${\mathbb C^n}$ \cite{MR3406539}.
But it is not proper in the case of the unit ball.
For example, 
Lemma \ref{SMVP} with the reproducing property and comparable property implies the following estimate for the Bergman kernel on diagonal:
\begin{align}\label{Hessian}
K_{\psi}(z,z) \leq C \frac{e^{2\psi(z)}}{(1-\lvert z \rvert^2)^{2n+1}}.
\end{align}
The estimate is same as the result which can be obtained from Theorem 3.3 in \cite{MR3632468} using series expansion. 
However, one could get only 
\begin{align}\label{Laplacian}
K_{\psi}(z,z) \leq C \frac{e^{2\psi(z)}}{(1-\lvert z \rvert^2)^{3n}}
\end{align} 
if one use the ball induced by a radius function with $\left(\Delta \psi\right)^{-\frac{1}{2}}$ instead of the $\psi$-Hessian  ball $B_{H}(z,r)$.
The estimate (\ref{Hessian}) is sharper than (\ref{Laplacian}) when $n>1$.  

We study integral operators on the exponential Bergman space using the Carleson embedding theorem. 
For Ces\`{a}ro operators on $A^p_{\psi}(\Bn)$ with holomorphic symbols, and Toeplitz operators on $A^2_{\psi}(\Bn)$ with symbols in $L^2_{\psi}(\Bn)$, 
equivalent conditions of boundedness and compactness are presented in Section \ref{sec4}. 
Note that we get the results on Toeplitz operators only in the case of $p=2$. 
We need more properties about the Bergman kernel for the result of $p\neq 2$. 
We have not yet acquired appropriate estimates of off-diagonal of the Bergman kernel which can give boundedness of the Bergman projection of $A^2_{\psi}(\Bn)$ (cf. \cite{MR3899961, MR2891634}).

For studying the Toeplitz operator, the test functions have essential roles.  
The Toeplitz operator with a symbol function $u \in L^2_\psi(\Bn)$, is defined 
\begin{align*}
  T_{u} f(z) = \int_{\Bn} K_{\psi}(z,w)f(w)u(w)e^{-\psi(w)}\diff v(w)
\end{align*}
for $f \in A^2_\psi(\Bn)$.
Let $\diff \mu= u \diff v$. Then we have 
\begin{align*}
  {\widehat \mu_{2}} (z) =  \frac{1}{\|\Phi_{2,z}\|_{2,\psi}^2} \langle T_{u}\Phi_{2,z}, \Phi_{2,z} \rangle_{\psi},
\end{align*}
and its boundedness is equivalent to boundedness of the Toeplitz operator (see \eqref{imm_berezin} and Theorem \ref{bdd_Toeplitz}). The function ${\widehat \mu_{2}}$ behaves like the Berezin transform, which is defined with the Bergman kernel. Precisely, the test function $\Phi_{2,z}$ is used instead of the Bergman kernel function in typical methods. 

\bigskip

Throughout this paper, $C$ will be a symbol of a positive constant. The value of the constant can be changed often. The expression $A \lesssim B$ indicates $A\leq C B$, and $A \simeq B$ means that $A \lesssim B$ and $B \lesssim A$.

\section{Preliminaries}\label{sec2}
\subsection{The ball induced by $H_{\psi}$}\label{Sec. 2.1}
Let 
$$
\psi(z)=\frac{1}{1-\lvert z \rvert^2},
$$
then $\psi$ is strictly plurisubharmonic. 
The \emph{complex Hessian} of $\psi$ is defined by 
$$
H_{\psi}:=\left(\frac{\partial^2{\psi}}{\partial z_j\partial\overline{z}_k}\right)_{n \times n}.
$$ 

\begin{lem}\label{lem for H} 
  For $\psi(z)=\frac{1}{1-\lvert z \rvert^2}$, the complex Hessian $H_{\psi}$ has the following properties:
 \begin{enumerate}
    \item [(a)] $ H_{\psi}(z) = \frac{1}{(1-\lvert z \rvert^2)^3}\left((1-\lvert z \rvert^2)I_{n\times n} + 2A(z)\right)$,
    where $A(z)=\left(\overline{z}_j z_k\right)_{jk}$;
    \item [(b)] $ H_{\psi}(z)^{-1} = {(1-\lvert z \rvert^2)^2}\left(I_{n\times n} - \frac{2}{1+\lvert z \rvert^2}A(z)\right)$;
    \item [(c)] $\det H_{\psi}(z) = \frac{1+\lvert z \rvert^2}{\left(1-\lvert z \rvert^2\right)^{2n+1}}$;
    \item [(d)] $H_{\psi}(z)= \frac{1+\lvert z \rvert^2}{(1-\lvert z \rvert^2)^3} P_z + \frac{1}{(1-\lvert z \rvert^2)^2}Q_z,$
    where $P_z \zeta = \frac{\langle \zeta, z\rangle}{\langle z, z\rangle} z$ for $z\in\Bn-\{0\}$, $P_0 = 0$, and $Q_z = I-P_z$.
\end{enumerate}
\end{lem}
\begin{proof}
Statements (a) and (d) are given by simple calculations.
Besides, the facts $A(z)=\lvert z \rvert^2 P_z$ and $P_z + Q_z = I$ are used for the case of (d).
  
  Let $u=(\overline{z}_1, \cdots, \overline{z}_n)$, $v= (z_1, \cdots ,z_n)$ be column vectors, then $u v^T=\left(\overline{z}_j z_k\right)_{jk}=A(z)$ and $v^T u = \lvert z \rvert^2$.
  From (a), it is obtained
  \begin{align}\label{Hessian matrix reform}
    H_{\psi}(z) = \frac{2}{(1-\lvert z \rvert^2)^3}\left(\frac{(1-\lvert z \rvert^2)}{2}I_{n\times n} + u v^T \right) = \frac{2}{(1-\lvert z \rvert^2)^3}\left(B(z) + u v^T \right)
  \end{align}
  where $B(z)=\frac{(1-\lvert z \rvert^2)}{2}I_{n\times n}$.
  The Sherman-Morrison formula \cite{MR35118} gives
  \begin{align*}
    \left(B(z) + u v^T \right)^{-1} 
    &= B(z)^{-1}-\frac{B(z)^{-1} uv^{T} B(z)^{-1}}{1+v^{T} B(z)^{-1} u}\\
    &= \frac{2}{1-\lvert z \rvert^2} I_{n\times n} - \frac{4}{(1+\lvert z \rvert^2)(1-\lvert z \rvert^2)}A(z)
  \end{align*}
  which shows (b).

  The matrix determinant lemma \cite{harville1998matrix} gives 
  \begin{align*}
    \det \left(B(z) + u v^T \right) 
    &= \left( 1 + v^{T} B(z)^{-1} u\right) \det B(z) \\
    &= \frac{(1 + \lvert z \rvert^2)(1-\lvert z \rvert^2)^{n-1}}{2^n}
  \end{align*}
  which provides (c) from \eqref{Hessian matrix reform}.
\end{proof}

For a piecewise $C^1$ curve $\gamma : [0,1] \rightarrow \Bn$, the length induced by Hessian metric is defined by
\begin{align*}
  \ell_{\psi}(\gamma):=  \int_{0}^{1} \sqrt{\langle H_{\psi}(\gamma(t))\gamma'(t),\gamma'(t)\rangle} \diff t.
\end{align*} 
For $z, w \in \Bn$, the distance induced by Hessian metric is
\begin{align*}
  \sigma(z,w):= \inf_\gamma \ell_{\psi}(\gamma) 
\end{align*} 
  where $\gamma$ is a parametrized curve from $z$ to $w$ in $\Bn$.

The $\psi$-Hessian ball centered at $z$ with radius $r>0$ is defined by the associated ball with $\sigma$, namely, 
\begin{align*}
  B_{H}(z, r) := \left\{w \, ; \, \sigma(z,w) <r\right\}.
\end{align*}

Since $P_z$ is the orthogonal projection of ${\mathbb C^n}$ onto the subspace $[z]$ generated by $z$, and $Q_z = I-P_z$ is the projection onto the orthogonal complement of $[z]$, the statement (d) in Lemma \ref{lem for H} means that for $n \geq 2$ and $z \neq 0$, the matrix $ H_{\psi}(z)$ has two eigenvalues, namely, $\frac{1+\lvert z \rvert^2}{(1-\lvert z \rvert^2)^3}$ with eigenspace $[z]$, and $\frac{1}{(1-\lvert z \rvert^2)^2}$ with eigenspace ${\mathbb C^n} \ominus [z]$. That brings the definition of another region:
\begin{align*}
  D_{\psi}(z,r) := \left\{w \, ; \, \lvert z-P_z w\rvert < r\left(1-\lvert z \rvert^2\right)^{\frac32}, \, \lvert Q_z w\rvert < r\left(1-\lvert z \rvert^2\right) \right\}.
\end{align*}

We also denote $B(z, r)$ for the Euclidean ball centered at $z$ with radius $r>0$. The following lemma is well known.
\begin{lem}[\cite{MR2679024}] \label{L1->radius}
  Let ${f} : \Bn \rightarrow {\mathbb R}$ be a positive function.
  If Lipschitz norm of the function ${f}$ is bounded, i.e.,
  \begin{align*}
    \left\|{f}\right\|_L :=\sup_{z,w\in \Bn, z \neq w} \frac{\lvert {f}(z)-{f}(w)\rvert }{\lvert z-w\rvert} <\infty, 
  \end{align*}
  then there is a constant $C>0$ satisfying 
  \begin{align*}
    \frac{1}{2} {f}(z) \leq {f}(w) \leq 2{f}(z) \quad \text{for} \quad w \in B(z,\kappa {f}(z))
  \end{align*}
  when $0 < \kappa \leq \rho_{f}$ and $\rho_{f} = \frac{1}{2}\min \left\{1, \frac{1}{\left\|{f}\right\|_L}\right\}.$
  \end{lem}
\begin{corollary}
    Let $f(z)=1-\lvert z \rvert^2$. Then $\left\|{f}\right\|_L=2$.
\end{corollary}

\begin{proof}
  Because $\lvert f(z)-f(w)\rvert = \lvert \lvert z \rvert^2-\lvert w \rvert^2 \rvert 
  \le (\lvert z \rvert+\lvert w \rvert)\lvert z-w\rvert,$
  we have
  \begin{align*}
    \left\|{f}\right\|_L=\sup_{z,w\in \Bn, z \neq w} \frac{\lvert f(z)-f(w)\rvert}{\lvert z-w\rvert}\le \sup_{z,w\in \Bn, z \neq w}(\lvert z \rvert+\lvert w \rvert)\le 2.
  \end{align*}

  For $\varepsilon>0$, let $z_{\varepsilon}=(1-\varepsilon,0,\cdots,0)$ and $w_{\varepsilon}=(1-2\varepsilon,0,\cdots,0)$. Then
  \begin{align*}
    \left\|{f}\right\|_L \ge \frac{\lvert {\lvert z_{\varepsilon}\rvert^2-\lvert w_{\varepsilon}\rvert^2} \rvert}{\lvert z_{\varepsilon}-w_{\varepsilon}\rvert}=2-3\varepsilon.
  \end{align*}
  Since $\varepsilon$ is arbitrary, $\left\|{f}\right\|_L \ge 2$.
  Hence $\left\|{f}\right\|_L = 2$. 
\end{proof}

\begin{lem}\label{comp in D}
  Let $0< r <\frac{C_0}{2}$ where $C_0 = \frac{1}{2} \min\left\{{1,\frac{1}{\|1-\lvert z \rvert^2\|_L}}\right\}=\frac{1}{4}
  $. Then
  $$1-\lvert z \rvert^2 \simeq 1-\lvert w \rvert^2$$ whenever $w \in D_{\psi}(z,r)$ for $z, w \in \Bn$.
\end{lem}
\begin{proof}
  Let $w \in D_{\psi}(z,r)$. Then
  $\lvert z- w\rvert \le \lvert z-P_z w\rvert+\lvert Q_z w\rvert < 2r(1-\lvert z \rvert^2)$
  which means $w$ belongs $B\left(z, 2r(1-\lvert z \rvert^2)\right)$. It implies $1-\lvert z \rvert^2 \simeq 1-\lvert w \rvert^2$ provided $2r < C_0$ by Lemma \ref{L1->radius}.
\end{proof}

We can obtain a geometric description of the $\psi$-Hessian ball $B_{H}(z, r)$ with $D_{\psi}(z, r)$ 
(see Section 7 in \cite{MR3010276}). For the completeness, we give the proof in detail.
\begin{proposition}\label{ball_size_comp1}
Let $0< r \le \frac{C_0}{20}$. Then
  \begin{align*}
        D_{\psi}\left(z, \frac{r}{10}\right) \subset B_{H}(z, r)
  \end{align*}
  where $C_0 = \frac{1}{2} \min\left\{{1,\frac{1}{\|1-\lvert z \rvert^2\|_L}}\right\}.$
\end{proposition}

\begin{proof}
  Recall that the distance induced by the metric $H_{\psi}$ is   
  \begin{align*}
    \sigma(z,w) = \inf_\gamma \ell_{\psi}(\gamma)
  \end{align*} 
  where the infimum is taken over all piecewise smooth curve ${\gamma}:[0,1] \rightarrow \Bn$ with $\gamma(0)=z$ and $\gamma(1)=w$ and the length induced by Hessian metric is 
  \begin{align*}
    \ell_{\psi}(\gamma) = \int_{0}^{1} \left\{ \frac{2}{\left(1-\lvert \gamma(t)\rvert^2\right)^{3}}\lvert \langle \gamma(t),\gamma'(t)\rangle\rvert^2 + \frac{1}{\left(1-\lvert \gamma(t)\rvert^2\right)^{2}} \lvert \gamma'(t)\rvert^2\right\}^{\frac{1}{2}} \diff t
  \end{align*} 
  for each curve $\gamma$.

  Suppose that $w$ belongs to $D_{\psi}(z,m)$ with $0<m<1$. We assume $w \neq z$ without loss of generality.
  We have $1-\lvert z \rvert^2 \simeq 1-\lvert w \rvert^2$ when $2m < C_0$ by Lemma \ref{comp in D}. 
  Let $\widehat{\gamma}_1$ be a line segment from $z$ to $P_z w$ and $\widehat{\gamma}_2$ be a line segment from $P_z w$ to $w$, 
precisely, 
$$
\widehat{\gamma}_1 (t)=(P_z w-z)t + z
$$ 
and 
$$
\widehat{\gamma}_2 (t)=(w-P_z w)t+P_z w=(Q_z w) t+P_z w.
$$ 
Let $\widehat{\gamma}$ be a parametrized curve for $\widehat{\gamma}_1 + \widehat{\gamma}_2$, then 
$$
\ell_{\psi}(\widehat{\gamma})=\ell_{\psi}(\widehat{\gamma}_1)+\ell_{\psi}(\widehat{\gamma}_2).
$$

  We have
  \begin{align*}
    \ell_{\psi}(\widehat{\gamma}_1) 
    &\le  \int_{0}^{1} \left\{\frac{\sqrt{2}}{\left(1-\lvert \widehat{\gamma}_1(t)\rvert^2\right)^{\frac{3}{2}}}\lvert \langle \widehat{\gamma}_1(t),\widehat{\gamma}_1'(t)\rangle\rvert + \frac{1}{1-\lvert \widehat{\gamma}_1(t)\rvert^2} \lvert \widehat{\gamma}_1'(t)\rvert\right\} \diff t \\ 
    &\le \frac{4}{\left(1-\lvert z \rvert^2\right)^{\frac{3}{2}}} \int_{0}^{1} \lvert \langle \widehat{\gamma}_1(t),\widehat{\gamma}_1'(t)\rangle\rvert \diff t + \frac{2}{1-\lvert z \rvert^2} \int_{0}^{1} \lvert \widehat{\gamma}_1'(t)\rvert \diff t 
  \end{align*}
  by Lemma \ref{comp in D}. 
  Note that $\widehat{\gamma}_1'(t) = P_z w - z $. The Cauchy-Schwartz inequality yields that 
  \begin{align*}
    \ell_{\psi}(\widehat{\gamma}_1) 
    \le& \frac{4}{\left(1-\lvert z \rvert^2\right)^{\frac{3}{2}}} \int_{0}^{1} \lvert \widehat{\gamma}_1(t)\rvert \lvert \widehat{\gamma}_1'(t)\rvert \diff t + 2\frac{\lvert z-P_z w\rvert}{1-\lvert z \rvert^2} \\ 
    \le& \frac{4}{\left(1-\lvert z \rvert^2\right)^{\frac{3}{2}}} \sup_{t} \lvert \widehat{\gamma}_1(t)\rvert \int_{0}^{1} \lvert \widehat{\gamma}'_1(t)\rvert \diff t + 2\frac{\lvert z-P_z w\rvert}{1-\lvert z \rvert^2}. 
  \end{align*}
  The fact $\sup_{t} \lvert \widehat{\gamma}_1(t)\rvert \le 1$ gives
  $$ \ell_{\psi}(\widehat{\gamma}_1) \le \frac{4}{\left(1-\lvert z \rvert^2\right)^{\frac{3}{2}}} \lvert z-P_z w\rvert+ 2\frac{\lvert z-P_z w\rvert}{1-\lvert z \rvert^2}
    < 6m
  $$ 
  when $w \in D_{\psi}(z,m).$

  For the length induced by $H_{\psi}$ of $\widehat{\gamma}_2$, we also have 
  \begin{align*}
    \ell_{\psi}(\widehat{\gamma}_2)
    &\le \frac{4}{\left(1-\lvert z \rvert^2\right)^{\frac{3}{2}}} \int_{0}^{1} \lvert \langle \widehat{\gamma}_2(t),\widehat{\gamma}_2'(t)\rangle\rvert \diff t + \frac{2}{1-\lvert z \rvert^2} \int_{0}^{1} \lvert \widehat{\gamma}_2'(t)\rvert \diff t
  \end{align*}
  by Lemma \ref{comp in D}. Note that $\widehat{\gamma}_2'(t) = Q_z w$.
  The fact that $P_z w$ and $Q_z w$ are perpendicular asserts
$$
\langle \widehat{\gamma}_2(t),\widehat{\gamma}_2'(t)\rangle = \left\langle (Q_z w)t+P_z w, Q_z w\right\rangle = t \lvert Q_z w\rvert^2.
$$
  Then, 
  \begin{align*}
    \ell_{\psi}(\widehat{\gamma}_2) 
    \le& \frac{4}{\left(1-\lvert z \rvert^2\right)^{\frac{3}{2}}} \lvert Q_z w\rvert^2 \int_{0}^{1} t \diff t + 2\frac{\lvert Q_z w\rvert}{1-\lvert z \rvert^2}\\
    \le& \frac{2}{\left(1-\lvert z \rvert^2\right)^{\frac{3}{2}}} \lvert Q_z w\rvert^2 + 2\frac{\lvert Q_z w\rvert}{1-\lvert z \rvert^2} \\
    <& 4m
  \end{align*} 
  when $w \in D_{\psi}(z,m).$

  As a result, we get $$\sigma(z,w) \le \ell_{\psi}(\widehat{\gamma}) < 10m$$ which implies $w\in  B_{H}(z, 10m)$. By putting $m=\frac{r}{10}$, it is obtained   
  \begin{align*}
    D_{\psi}\left(z, \frac{r}{10}\right) \subset B_{H}(z, r).
  \end{align*}
\end{proof}

\begin{proposition}\label{ball_size_comp2}
Let $0< r \le \frac{C_0}{20}$. Then
  \begin{align*}
    B_{H}(z, r) \subset D_{\psi}(z, 18r)
  \end{align*}
  where $C_0 = \frac{1}{2} \min\left\{{1,\frac{1}{\|1-\lvert z \rvert^2\|_L}}\right\}=\frac{1}{4}.$
\end{proposition}

\begin{proof}
  Suppose that $w$ belongs to $B_{H}(z, r)$. We assume $w \neq z$ without loss of generality. The proof is divided into three steps.

  {\bf Step 1.} We will show that $\sigma(z, w) <r$ implies $\lvert z-w\rvert \le 2r(1-\lvert z \rvert^2)$ and $\lvert Q_z(w)\rvert \le 2r(1-\lvert z \rvert^2)$. 

  Suppose $\sigma(z, w) <r$. As in Proposition 5 in \cite{MR3406539},
  let 
  \begin{align*}
    s=\frac{\lvert z-w\rvert}{1-\lvert z \rvert^2}.
  \end{align*}
  For any piecewise $C^1$ curve $\gamma$ from $z$ to $w$, let $T_0 \in (0,1]$ be the minimum of $t$ satisfying 
  \begin{align*}
    \lvert z- \gamma(t)\rvert = \min\left\{s,\frac{C_0}{10}\right\}(1-\lvert z \rvert^2) 
  \end{align*}
  where $C_0 = \frac{1}{2} \min\left\{{1,\frac{1}{\|1-\lvert z \rvert^2\|_L}}\right\}.$
  It gives
  \begin{align*}
    \frac{1}{2}(1-\lvert z \rvert^2) \le 1-\lvert \gamma(t)\rvert^2 \le 2(1-\lvert z \rvert^2) \quad \text{for}\quad t \in [0,T_0]
  \end{align*} 
  by Lemma \ref{L1->radius}. We have
  \begin{align}
    \begin{split}\label{Eucli}
    \ell_{\psi}(\gamma) 
    &\ge \int_{0}^{1} \frac{1}{1-\lvert \gamma(t)\rvert^2} \lvert \gamma'(t)\rvert \diff t\\
    &\ge \int_{0}^{T_0} \frac{1}{1-\lvert \gamma(t)\rvert^2} \lvert \gamma'(t)\rvert \diff t\\
    &\ge \frac{1}{2(1-\lvert z \rvert^2)} \int_{0}^{T_0} \lvert \gamma'(t)\rvert \diff t.
    \end{split}
  \end{align}
  Since $\int_{0}^{T_0} \lvert \gamma'(t)\rvert \diff t$ is the Euclidean length of the curve $\gamma$ from $0$ to $T_0$, $$\ell_{\psi}(\gamma) \ge \frac{1}{2} \min\left\{s,\frac{C_0}{10}\right\}.$$ 
  Because \eqref{Eucli} holds for arbitrary $\gamma$ connecting $z$ and $w$, 
  \begin{align*}
    \sigma(z,w) \ge \frac{1}{2} \min\left\{s,\frac{C_0}{10}\right\}.
  \end{align*}
  By the assumption $w \in B_{H}(z, r)$, it is obtained
  \begin{align*}
    \frac{C_0}{20} \ge r > \sigma(z,w) \ge \frac{1}{2} \min\left\{{s,\frac{C_0}{10}}\right\} 
  \end{align*}
  which is a contradiction whenever $s\ge \frac{C_0}{10}$. 
  Hence we have $s < \frac{C_0}{10}$ when $w \in B_{H}(z, r)$. 
  It gives $$r > \min\left\{{s,\frac{C_0}{10}}\right\} =\frac{1}{2} s,$$ 
  which asserts
  \begin{align}\label{B_H radius}
    \lvert z-w\rvert < 2r(1-\lvert z \rvert^2) \quad \text{for} \quad w \in B_{H}(z, r).
  \end{align}

  Also, \eqref{B_H radius} yields
  \begin{align}\label{concl_step1}
    \lvert Q_z w\rvert \le \lvert z-w\rvert<2r(1-\lvert z \rvert^2) \quad \text{for} \quad w \in B_{H}(z, r)
  \end{align}
  since $z-P_z w$, $z-w$, and $Q_z w$ construct a right triangle with hypotenuse $z-w$.
  ~\\

  Now, we will show that $\sigma(z, w) <r$ implies $\lvert z-P_z w\rvert \lesssim r(1-\lvert z \rvert^2)^{\frac{3}{2}}$. It divides into the cases of $\lvert z \rvert \le \frac{1}{2}$ and $\lvert z \rvert > \frac{1}{2}$. 

  {\bf Step 2.} 
  We assume that $\lvert z \rvert \le \frac{1}{2}$. 
  \eqref{B_H radius} in Step 1 gives $$B_{H}(z, r) \subset B(z, 2r)$$ where $B(z,r)$ is the Euclidean ball centered at $z$ with radius $r>0$. 
  Since $\lvert z \rvert\le \frac{1}{2}$ means $\frac{3}{4} \le 1-\lvert z \rvert^2 \le 1$, it is obtained
  $$\lvert z-w\rvert <2r = \frac{16}{\sqrt{27}} r \left(\frac{3}{4}\right)^{\frac{3}{2}} \le \frac{16}{\sqrt{27}} r (1-\lvert z \rvert^2)^{\frac{3}{2}} <4r(1-\lvert z \rvert^2)^{\frac{3}{2}}.$$ 
  Because $z-P_z w$, $z-w$, and $Q_z w$ construct a right triangle with hypotenuse $z-w$, we have $\lvert z-P_z w\rvert \le \lvert z-w\rvert$. 

  Therefore,
  \begin{align}\label{concl_step2}
    \lvert z-P_z w\rvert <4r(1-\lvert z \rvert^2)^{\frac{3}{2}}  \quad \text{for} \quad w \in B_{H}(z, r)
  \end{align}
  when $\lvert z \rvert \le \frac{1}{2}$.

  {\bf Step 3.} 
  We assume that $\lvert z \rvert > \frac{1}{2}$.
  Suppose that $\sigma(z, w)<r$. 
  We also assume $z \neq P_z w$ without loss of generality. 
  Hereinafter, we consider only the curves $\gamma$ connecting $z$ and $w$ satisfying $$\ell_{\psi}(\gamma) \le 2\sigma(z, w).$$ 
  For each curve $\gamma$, define $\gamma_1(t)=P_z(\gamma(t))$ and $\gamma_2(t)=\gamma(t)-\gamma_1(t)= Q_z(\gamma(t))$.
  Let 
  $t_0 \in [0,1]$ be the minimum of $t$ such that 
  \begin{align*}
    \lvert z- \gamma_1(t)\rvert = \lvert z-P_z w\rvert. 
  \end{align*}
  Similar to Lemma \ref{L1->radius}, we have comparable property in the ball $B_{H}(z, r)$ by \eqref{B_H radius}, precisely,
  \begin{align}\label{comp in B_H}
    \frac{1}{2}(1-\lvert z \rvert^2) \le 1-\lvert w \rvert^2 \le 2(1-\lvert z \rvert^2) \quad \text{for} \quad w \in B_{H}(z, r). 
  \end{align} 
  Thus, we have
  \begin{align*}
    \ell_{\psi}(\gamma) 
    &\ge \int_{0}^{1} \frac{\sqrt{2}}{\left(1-\lvert \gamma(t)\rvert^2\right)^{\frac{3}{2}}}\lvert \langle \gamma(t),\gamma'(t)\rangle\rvert \diff t\\
    &\ge \frac{1}{2\left(1-\lvert z \rvert^2\right)^{\frac{3}{2}}} \int_{0}^{t_0} \lvert \langle \gamma(t),\gamma'(t)\rangle\rvert \diff t\\
    &\ge \frac{1}{2\left(1-\lvert z \rvert^2\right)^{\frac{3}{2}}} \int_{0}^{t_0} \left\{\lvert \langle \gamma_1(t),\gamma_1'(t)\rangle\rvert-\lvert \langle \gamma_2(t),\gamma_2'(t)\rangle\rvert\right\} \diff t.
  \end{align*}

  Since $\gamma_1$ and $\gamma_1'$ are parallel, 
  \begin{align}\label{gamma1}
    \int_{0}^{t_0} \lvert \langle \gamma_1(t),\gamma_1'(t)\rangle\rvert \diff t =\int_{0}^{t_0} \lvert \gamma_1(t)\rvert\lvert \gamma_1'(t)\rvert \diff t \ge \inf_{0\le t \le t_0} \lvert \gamma_1(t)\rvert \int_{0}^{t_0} \lvert \gamma_1'(t)\rvert \diff t. 
  \end{align}
  The hypothesis $\ell_{\psi}(\gamma) \le 2\sigma(z, w)$ gives 
  $$\sigma(z, \gamma_1 (t)) < \sigma(z, \gamma (t)) < 2r.$$ 
  Inequality \eqref{B_H radius} and $0 < 1-\lvert z \rvert^2 < \frac{3}{4}$ yield
  $$\gamma_1 (t) \in B_{H} (z, 2r) \subset B (z, 4r(1-\lvert z \rvert^2)) \subset B (z, 3r).$$
  It gives $$\inf_{0\le t \le t_0} \lvert \gamma_1(t)\rvert \ge \lvert z \rvert - 3r > \frac{1}{2} - \frac{3}{80} =\frac{37}{80}$$
  with $0<r\le \frac{C_0}{20} = \frac{1}{80}$.
  Hence, we have $$\int_{0}^{t_0} \lvert \langle \gamma_1(t),\gamma_1'(t)\rangle\rvert \diff t 
  \ge \frac{37}{80} \lvert z-P_z w\rvert$$
  from \eqref{gamma1} and $\int_{0}^{t_0} \lvert \gamma_1'(t)\rvert \diff t \ge \lvert z-P_z w\rvert$.

  By Cauchy-Schwartz inequality,
  \begin{align*}
    \int_{0}^{t_0} \lvert \langle \gamma_2(t),\gamma_2'(t)\rangle\rvert \diff t
    \le \int_{0}^{t_0} \lvert  \gamma_2(t) \rvert \lvert  \gamma_2'(t)\rvert \diff t
    \le  \sup_{0\le t \le t_0} \lvert  \gamma_2(t) \rvert \int_{0}^{t_0} \lvert  \gamma_2'(t)\rvert \diff t.
  \end{align*}
  Let $t^* \in [0,t_0]$ satisfy 
$$
\sup_{0\le t \le t_0} \lvert \gamma_2(t) \rvert = \lvert \gamma_2(t^*)\rvert = 
\lvert Q_z(\gamma(t^*))\rvert.
$$
  Since $z-\gamma_1(t^*)=z-P_z(\gamma(t^*))$, $z-\gamma(t^*)$ 
and $\gamma_2(t^*)=Q_z(\gamma(t^*))$ 
construct a right triangle with hypotenuse $z-\gamma(t^*)$, we have
  $$\sup_{0\le t \le t_0} \lvert  \gamma_2(t) \rvert = \lvert Q_z (\gamma(t^*)) \rvert \le \lvert z-\gamma(t^*)\rvert \le \int_{0}^{t*} \lvert  \gamma'(t)\rvert \diff t \le \int_{0}^{t_0} \lvert  \gamma'(t)\rvert \diff t. $$ 
  Then, we get
  \begin{align*}
    \int_{0}^{t_0} \lvert \langle \gamma_2(t),\gamma_2'(t)\rangle\rvert \diff t
    \le \left\{\int_{0}^{t_0} \lvert  \gamma'(t)\rvert \diff t\right\}^2
    \le 4(1-\lvert z \rvert^2)^2 \ell_{\psi}(\gamma)^2
  \end{align*}
  since $\int_{0}^{t_0} \lvert \gamma'(t)\rvert \diff t \le 2(1-\lvert z \rvert^2) \ell_{\psi}(\gamma) $ as in \eqref{Eucli}.

  Therefore, we obtain  
  \begin{align*}
    \ell_{\psi}(\gamma)
    &\ge \frac{37}{160} \frac{1}{\left(1-\lvert z \rvert^2\right)^{\frac{3}{2}}} \lvert z-P_z w\rvert - 2\left(1-\lvert z \rvert^2\right)^{\frac{1}{2}} \ell_{\psi}(\gamma)^2
  \end{align*}
  which implies
  \begin{align*}
    \frac{37}{160} \frac{1}{\left(1-\lvert z \rvert^2\right)^{\frac{3}{2}}} \lvert z-P_z w\rvert
    \le 2 \ell_{\psi}(\gamma)
    \le 4 \sigma(z,w)
    < 4r.
  \end{align*}
  It shows
  \begin{align}\label{concl_step3}
    \lvert z-P_z w\rvert < 18r \left(1-\lvert z \rvert^2\right)^{\frac{3}{2}} \quad \text{for} \quad w \in B_{H}(z, r)
  \end{align}
  when $\lvert z \rvert > \frac{1}{2}$.

  Finally, we get the desired result 
  $$
  B_{H}(z, r) \subset D_{\psi}(z, 18r) 
  $$ 
  by gathering with \eqref{concl_step1}, \eqref{concl_step2}, and \eqref{concl_step3}. 
\end{proof}

From Propositions \ref{ball_size_comp1} and \ref{ball_size_comp2}, we have the following theorem.
\begin{theorem}\label{ball_size_comp}
   For $0< r \le \frac{C_0}{20}$, there is a constant $\delta >1$ depending only on $C_0$ satisfying
  \begin{align*}
    D_{\psi}(z, \delta^{-1}r) \subset B_{H}(z, r) \subset D_{\psi}(z, \delta r)
  \end{align*}
  where $C_0 = \frac{1}{2} \min\left\{{1,\frac{1}{\|1-\lvert z \rvert^2\|_L}}\right\} = \frac{1}{4}.$
\end{theorem}

\begin{corollary}\label{cor_vol}
  For $0< r_1, r_2 \le \frac{C_0}{20}$, we have
  \begin{align*}
    v(B_{H}(z, r_1)) \simeq v(D_{\psi}(z,r_2)) \simeq \left(1-\lvert z \rvert^2\right)^{2n+1}.
  \end{align*}
  where $C_0 = \frac{1}{2} \min\left\{{1,\frac{1}{\|1-\lvert z \rvert^2\|_L}}\right\} = \frac{1}{4}.$
\end{corollary}

By triangle inequality of the distance $\sigma$, we get the following lemma.

\begin{lem}\label{z w inclusion}
  Given $r>0$, there are $r_1, r_2 > 0$ such that
  $$B_H(w,r_1 ) \subset B_H(z,r) \subset B_H(w,r_2)$$ whenever $\sigma(z,w)<r.$
\end{lem}

We have covering lemmas for $\psi$-Hessian balls 
by the same way in Lemmas 2.22 and 2.23 in \cite{MR2115155} using Lemma \ref{z w inclusion}.

\begin{lem}
  Let $R$ be a positive number and $m$ be a positive integer. Then there exists a positive integer $N$ such that every ball $B_{H}(z, r)$ with $r \le R$ can be covered by $N$ balls $B_{H}(a_k, \frac{r}{m})$.
\end{lem}

\begin{lem}\label{covering lem 2}
  There is a positive integer $N$ such that for any $0 < r \le 1$ we can find a sequence $\{a_k\}$ in $\Bn$ with the following properties:
  \begin{itemize}
    \item[(1)] $\Bn = \cup_k B_{H} (a_k, r)$.
    \item[(2)] The sets $B_{H} (a_k, r/4)$ are mutually disjoint.
    \item[(3)] Each point $z \in \Bn$ belongs to at most $N$ of the sets $B_{H} (a_k, 4r)$.
  \end{itemize}
\end{lem}

We say that a sequence $\{a_k\}$ is a $H_{\psi}$-lattice with $r$ when $\{a_k\}$ is a sequence satisfying the properties in Lemma \ref{covering lem 2}. 

\subsection{Test functions}

For $z \in \Bn$, the involutive automorphisms  on $\Bn$ are defined 
\begin{align*}
 \varphi_{z}(w):=\frac{z-P_z w -\sqrt{1-\lvert z \rvert^2}Q_z w}
{1-\langle w, z \rangle}.   
\end{align*} 
It has the following property:
\begin{align}\label{prop4_phi}
  1-\lvert \varphi_{z} (w)\rvert^2 = \frac{(1-\lvert z \rvert^2)(1-\lvert w \rvert^2)}{\lvert 1-\langle w, z\rangle \rvert^2}.
\end{align}
For more details of the automorphisms of $\Bn$, see page 23 of \cite{MR601594} or page 3 of \cite{MR2115155}. 

Due to the definition of $D_{\psi}(z,r)$, we can get the following inequality which is essential for proving the first estimate of the test functions in Lemma \ref{test fn lem}. 

\begin{lem}\label{imp_ineq}
  For $z \in \Bn$ and small $r > 0$, there is a constant $C$ depending only on the radius $r$ satisfying
  \begin{align*}
    \lvert 2 \text{\rm Re} \left(\frac{1}{1- \langle w, z \rangle }\right) -\frac{1}{1-\lvert z \rvert^2} -\frac{1}{1-\lvert w \rvert^2}\rvert \le C
  \end{align*}
  for $w \in D_{\psi}(z, r)$. 
\end{lem}

\begin{proof}
  Using \eqref{prop4_phi}, we get the reformulation:
  \begin{align}\label{eq_difference}
    2\re &\left(\frac{1}{1- \langle w, z \rangle }\right) -\frac{1}{1-\lvert z \rvert^2} -\frac{1}{1-\lvert w \rvert^2} \nonumber\\
    & = \frac{\lvert z-w\rvert^2}{\lvert 1- \langle w ,z\rangle \rvert^2} - \lvert \varphi_z (w)\rvert^2 \left(\frac{1}{1-\lvert z \rvert^2} + \frac{1}{1-\lvert w \rvert^2}\right),
  \end{align}
  which indicates 
    \begin{align}\label{Eq}
    - \lvert \varphi_z (w)\rvert^2 \left( \frac{1}{1-\lvert z \rvert^2} + \frac{1}{1-\lvert w \rvert^2}\right) \le \text{LHS of \eqref{eq_difference}} \le \frac{\lvert z-w\rvert^2}{\lvert 1- \langle w ,z\rangle \rvert^2}.
  \end{align}

  First, we show that $\frac{\lvert z-w\rvert^2}{\lvert 1- \langle w ,z\rangle \rvert^2}$ is dominated with some constant independent of $z$ and $w$, 
which means the LHS of \eqref{eq_difference} has an upper bound $C_r$. 
  For $w\in D_{\psi}(z,r)$, we have $\lvert z-w\rvert^2 < 4r^2(1-\lvert z \rvert^2)^2$ since $\lvert z-w\rvert \le \lvert z-P_z (w)\rvert+\lvert Q_z w\rvert$.

By Lemma \ref{comp in D}, we have $1-\lvert z \rvert^2 \simeq 1-\lvert w \rvert^2$ for $w \in D_{\psi}(z,r)$ for small $r>0$.
Hence there exists $C_r>0$ such that
  \begin{align}\label{upper constant 1}
    \frac{\lvert z-w\rvert^2}{\lvert 1- \langle w ,z\rangle \rvert^2} < \frac{4 r^2(1-\lvert z \rvert^2)^2}{\lvert 1- \langle w ,z\rangle \rvert^2} \le C_r \frac{(1-\lvert z \rvert^2)(1-\lvert w \rvert^2)}{\lvert 1- \langle w ,z\rangle \rvert^2}.
  \end{align}
  The RHS of \eqref{upper constant 1} is equal to $$C_r \left(1-\lvert \varphi_{z} (w)\rvert^2\right)$$ by \eqref{prop4_phi}. It is dominated by $C_r$ since $\varphi_{z} (w)$ belongs to $\Bn$. 

  Next, we show that the LHS of \eqref{Eq} has a lower bound $- C_r'$. Since $z-P_z (w)$ and $Q_z (w)$ are perpendicular, we have $$ \lvert \varphi_z(w)\rvert^2 = \frac{\lvert z-P_z (w)\rvert^2 + (1-\lvert z \rvert^2) \lvert Q_z (w)\rvert^2}{\lvert 1- \langle w ,z\rangle \rvert^2}. $$
  The definition of $D_{\psi}(z, r)$ yields
  $$\lvert \varphi_z (w)\rvert^2 < \frac{2r^2(1-\lvert z \rvert^2)^3}{\lvert 1- \langle w ,z\rangle \rvert^2}$$ 
  for $w \in D_{\psi}(z, r)$.
  It implies 
  \begin{align*}
    \lvert \varphi_z (w)\rvert^2 \left( \frac{1}{1-\lvert z \rvert^2} + \frac{1}{1-\lvert w \rvert^2}\right)
    &< \frac{2r^2(1-\lvert z \rvert^2)^3}{\lvert 1- \langle w ,z\rangle \rvert^2}\left( \frac{1}{1-\lvert z \rvert^2} + \frac{1}{1-\lvert w \rvert^2}\right)\\
    &\le C_r' \frac{(1-\lvert z \rvert^2)(1-\lvert w \rvert^2)}{\lvert 1- \langle w ,z\rangle \rvert^2} \le C_r'
  \end{align*}
  with Lemma \ref{comp in D}.
  The proof is done by getting $C=\max \left\{C_r , C_r'\right\}$.
\end{proof}

By the previous inequality, we can get the following lemma for test functions. It will be used for proving the weighted sub-mean-value property and the results in Section 3 and Section 4.  

\begin{lem}\label{test fn lem}
  For $z \in \Bn$, let $\Phi_{p,z} (w) := e^{\frac{2}{p} \frac{1}{1-\langle w,z\rangle}-\frac{1}{p}\frac{1}{1-\lvert z \rvert^2}}$. The holomorphic function $\Phi_{p,z}$ has following properties:
  \begin{align}\label{test fn_1}
    \lvert \Phi_{p,z} (w)\rvert^p e^{-\frac{1}{1-\lvert w \rvert^2}} \simeq 1 \quad \text{when} \quad w \in D_{\psi}(z, r) 
  \end{align}
  and
  \begin{align}\label{test fn_2}
    \|\Phi_{p,z}\|_{p,\psi}^p \simeq (1-\lvert z \rvert^2)^{2n+1}. 
  \end{align}
\end{lem}

\begin{proof}
  By Lemma \ref{imp_ineq}, we can show that for $z \in \Bn$ and small $r > 0$, there is a constant $C$ depending only $r$ satisfying
  \begin{align*}
    {C}^{-1} e^{- \frac{1}{1-\lvert z \rvert^2} - \frac{1}{1-\lvert w \rvert^2}} \le \lvert  e^{-\frac{1}{1- \langle w, z \rangle }}\rvert^{2} \le C e^{- \frac{1}{1-\lvert z \rvert^2} - \frac{1}{1-\lvert w \rvert^2}}
  \end{align*}
  for $w \in D_{\psi}(z, r)$. It gives \eqref{test fn_1}.
  By Lemma 3 in \cite{MR2371433}, \eqref{test fn_2} is obtained.
\end{proof}

\subsection{Mean value inequality with exponential weight}

\begin{lem}\label{SMVP}
  Let $f$ be a holomorphic function on $\Bn$ and $s \in {\mathbb R}$. 
  For $z \in \Bn$ and a sufficiently small radius $r>0$, there is a constant $C$ depending on $s$ and $r$ satisfying 
  \begin{align*}
  \lvert f(z)\rvert^p e^{-\frac{s}{1-\lvert z \rvert^2}} \le \frac{C}{(1-\lvert z \rvert^2)^{2n+1}} \int_{B_{H}(z,r)} \lvert f(\zeta)\rvert^p e^{-\frac{s}{1-\lvert \zeta\rvert^2}} \diff v(\zeta).  
  \end{align*}
\end{lem}

\begin{proof}
  Since the function $\Phi_{p,z}$ is non-vanishing, $\Phi_{p,z}(\zeta)^{-s}$ with a principle branch is holomorphic. 
  Subharmonicity of $\lvert f(\zeta)\Phi_{p,z}(\zeta)^{-s} \rvert^p$ gives that for $\delta > 1$,
  \begin{align*}
    \lvert f(z)\Phi_{p,z}(z)^{-s}\rvert^p 
    &\leq \frac{C}{v(D_{\psi}(z,\delta^{-1}r))} \int_{D_{\psi}(z,\delta^{-1}r)} \lvert f(\zeta)\rvert^p \lvert \Phi_{p,z}(\zeta)^{-p}\rvert^{s} \diff v(\zeta)\\
    &\le \frac{C}{v(D_{\psi}(z,\delta^{-1}r))} \int_{D_{\psi}(z,\delta^{-1}r)} \lvert f(\zeta)\rvert^p e^{-\frac{s}{1-\lvert \zeta\rvert^2}} \diff v(\zeta)
  \end{align*}
  with aid of \eqref{test fn_1}.
    We note that  $\lvert f(z)\rvert^p e^{-\frac{s}{1-\lvert z \rvert^2}} = \lvert f(z)\Phi_{p,z}(z)^{-s}\rvert^p$.
  Theorem \ref{ball_size_comp} and Corollary \ref{cor_vol} yield
  \begin{align*}
    \lvert f(z)\rvert^p e^{-\frac{s}{1-\lvert z \rvert^2}} 
    &\le \frac{C}{(1-\lvert z \rvert^2)^{2n+1}} \int_{B_{H}(z,r)} \lvert f(\zeta)\rvert^p e^{-\frac{s}{1-\lvert \zeta\rvert^2}} \diff v(\zeta).
  \end{align*}
\end{proof}

\begin{rmk}
When $n>1$, using the $\psi$-Hessian ball is suitable for investigating exponential type weighted Bergman spaces on the unit ball rather than using the ball with radius function $\left(\Delta \psi\right)^{-\frac{1}{2}}$. 
However, when $n=1$, it follows that
$H_\psi (z) = \Delta \psi (z) \simeq (1-\lvert z \rvert^2)^{-3}$.
Moreover, note that $P_z = I$ for $z\in{\mathbb D}-\{0\}$, 
and $Q_z \equiv 0$.
Hence $D_{\psi}(z,r)$, the $\psi$-Hessian ball $B_H(z,r)$, and
the ball with radius function $\left(\Delta \psi\right)^{-\frac{1}{2}}$ are all comparable when $n=1$.
\end{rmk}

\section{Carleson embedding theorem}\label{sec3}

\begin{definition}
  For a Borel measure $\diff \mu$, we call the measure $\diff \mu$ is a Carleson measure for $A^p_{\psi}(\Bn)$ if there is a constant $C>0$ satisfying
    \begin{align*}
      \int_{\Bn} \lvert  f(z) \rvert^p e^{-\psi(z)} \diff\mu(z) \le C \int_{\Bn} \lvert  f(z) \rvert^p e^{-\psi(z)} \diff v(z).
    \end{align*}
\end{definition}

We denote
\begin{align*}
  \|f\|_{p,\mu} := \left\{\int_{\Bn} \lvert  f(z) \rvert^p e^{-\psi(z)} \diff\mu(z)\right\}^{\frac{1}{p}}
\end{align*}
as the norm of $f$ which belong to $L^p(\Bn, e^{- \psi} \diff \mu )$.

Let $p \ge 1$. For $z \in \Bn$, let $$\widetilde{\Phi}_{p,z} (w) := \frac{\Phi_{p,z}(w)}{\|\Phi_{p,z}\|_{p,\psi}}$$ be the normalized test function in Lemma \ref{test fn lem}.
Then $\widetilde{\Phi}_{p,z} \in A^p_{\psi}(\Bn)$ and 
\begin{align}\label{normalized test function}
  \widetilde{\Phi}_{p,z}(w) \simeq \frac{\Phi_{p,z} (w)}{(1-\lvert z \rvert^2)^{\frac{2n+1}{p}}} = \frac{e^{\frac{2}{p}\frac{1}{1-\langle w,z\rangle} -\frac{1}{p}\frac{1}{1-\lvert z \rvert^2}}}{(1-\lvert z \rvert^2)^{\frac{2n+1}{p}}}
\end{align}
by \eqref{test fn_2}.
We note that $\widetilde{\Phi}_{p,z}(w)$ converges to $0$ uniformly on compact subsets of $\Bn$ as $\lvert z \rvert \rightarrow 1^{-}$.

For a finite positive Borel measure $\mu$, we define a function ${\widehat \mu}_p$ with the normalized test functions in $A^p_{\psi}(\Bn)$; 
\begin{align*}
  {\widehat \mu}_p (z) := \int_{\Bn} \lvert\widetilde{\Phi}_{p,z}(w)\rvert^p e^{-\psi(w)}\diff \mu(w) = \frac{1}{\|\Phi_{p,z}\|_{p,\psi}^p} \int_{\Bn} \lvert\Phi_{p,z}(w)\rvert^p e^{-\psi(w)}\diff \mu(w).
\end{align*}

\begin{theorem}\label{Carl measure char}
  Let $\diff\mu$ be a positive Borel measure. The following statements are equivalent:
  \begin{enumerate}
    \item[(a)] The measure $\diff\mu$ is a Carleson measure for $A^p_{\psi}(\Bn)$;
    \item[(b)] ${\widehat \mu}_p$ is a bounded function on $\Bn$;
    \item[(c)] For $z \in \Bn$ and sufficiently small $r>0$, there is a constant $C>0$ satisfying
    \begin{align*}
      \mu(B_{H}(z,r)) \le C v(B_{H}(z,r))\textit{;}
    \end{align*}
    \item[(d)] For any $H_\psi$-lattice $\{a_k\}$, there is a constant $C>0$ satisfying
    \begin{align*}
      \mu(B_{H}(a_k,r)) \le C v\!\left(B_{H}(a_k,r)\right).
    \end{align*}
  \end{enumerate}
\end{theorem}

\begin{proof}
  Suppose $\diff\mu$ is a Carleson measure for $A^p_{\psi}(\Bn)$. There is $C>0$ satisfying 
  \begin{align*}
    \int_{\Bn} \lvert \Phi_{p,z}(w)\rvert^p e^{-\psi(w)} \diff\mu(w) 
    \le C\int_{\Bn} \lvert \Phi_{p,z}(w)\rvert^p e^{-\psi(w)} \diff v(w)
    = C\|\Phi_{p,z}\|_{p,\psi}^p
  \end{align*}  
  for $z \in \Bn$. It gives (a) implies (b).

  Using Theorem \ref{ball_size_comp} and Lemma \ref{test fn lem}, we have
  \begin{align*}
    \mu(B_{H}(z,r)) = \int_{B_{H}(z,r)} \diff\mu(w) 
    \simeq \int_{B_{H}(z,r)} \lvert \Phi_{p,z}(w)\rvert^p e^{-\psi(w)} \diff\mu(w).
  \end{align*}
  Corollary \ref{cor_vol} and \eqref{test fn_2} give 
  \begin{align*}
    v(B_{H}(z,r)) \simeq (1-\lvert z \rvert^2)^{2n+1} \simeq \|\Phi_{p,z}\|_{p,\psi}^p.
  \end{align*}
  Therefore,
  \begin{align}\label{char_Carl_b->c}
    \frac{\mu(B_{H}(z,r))}{v(B_{H}(z,r))} \lesssim {\widehat \mu}_p (z)
  \end{align}
  which shows (b) implies (c).

  The implication (c) $\Rightarrow$ (d) is trivial.

  Suppose (d) holds. By Lemma \ref{covering lem 2}, we have
  \begin{align}\label{1st for CMC}
    \int_{\Bn} \lvert f(z)\rvert^p e^{-\psi(z)} \diff\mu(z) \le \sum_{k=1}^{\infty} \int_{B_H(a_k, r)} \lvert f(z)\rvert^p e^{-\psi(z)} \diff\mu(z).
  \end{align}
  Since Lemma \ref{SMVP} asserts
  \begin{align*}
    \sup_{z\in B_H(a_k, r)} \lvert f(z)\rvert^p e^{-\psi(z)} \lesssim \frac{1}{v\!\left(B_{H}(a_k,r)\right)} \int_{B_H(a_k, 2r)} \lvert f(w)\rvert^p e^{-\psi(w)} \diff v(w), 
  \end{align*}
  it is obtained that
  \begin{align*}
    \text{LHS of } \eqref{1st for CMC} &\lesssim \sum_{k=1}^{\infty} \int_{B_H(a_k, r)} \frac{1}{v\!\left(B_{H}(a_k,r)\right)} \int_{B_H(a_k, 2r)} \lvert f(w)\rvert^p e^{-\psi(w)} \diff v(w) \diff\mu(z)\\
    &= \sum_{k=1}^{\infty} \frac{1}{v\!\left(B_{H}(a_k,r)\right)} \int_{B_H(a_k, r)} \diff\mu(z) \int_{B_H(a_k, 2r)} \lvert f(w)\rvert^p e^{-\psi(w)} \diff v(w)\\
    &= \sum_{k=1}^{\infty} \frac{\mu(B_H(a_k, r))}{v\!\left(B_{H}(a_k,r)\right)} \int_{B_H(a_k, 2r)} \lvert f(w)\rvert^p e^{-\psi(w)} \diff v(w).
  \end{align*}
  By the hypothesis, there is $C>0$ such that
  \begin{align*}
    \text{LHS of } \eqref{1st for CMC} 
    &\leq C\sum_{k=1}^{\infty} \int_{B_H(a_k, 2r)} \lvert f(w)\rvert^p e^{-\psi(w)} \diff v(w).
  \end{align*}
  Lemma \ref{covering lem 2} yields 
  \begin{align*}
    \int_{\Bn} \lvert f(z)\rvert^p e^{-\psi(z)} \diff\mu(z) \leq CN\int_{\Bn} \lvert f(z)\rvert^p e^{-\psi(z)} \diff v(z)
  \end{align*}
  which implies (a).
\end{proof}
\begin{corollary}
  For $p \ge 1$ and a positive Borel measure $\diff \mu$, the following quantities are equivalent:
  \begin{enumerate}
    \item [(a)] $\|i_p\|^p$ where $\|i_p\| = \sup \left\{\|f\|_{p, \mu} ; \|f\|_{p, \psi} = 1\right\}$;
    \item [(b)] $\|{\widehat \mu}_p\|_{\infty}$;
    \item [(c)] For small $r>0$, $\sup_{z \in \Bn} \frac{\mu(B_{H}(z,r))}{v \left(B_{H}(z,r)\right)}$;
    \item [(d)] For any $H_\psi$-lattice $\{a_k\}$, $\sup_{k=1,2,\dots} \frac{\mu(B_{H}(a_k,r))}{v \left(B_{H}(a_k,r)\right)}$.
  \end{enumerate}
\end{corollary}
\begin{definition}
  For a Borel measure $\diff \mu$, we call the measure $\diff \mu$ is a vanishing Carleson measure for $A^p_{\psi}(\Bn)$ if 
  \begin{align*}
    \lim_{j \rightarrow \infty} \int_{\Bn} \lvert  f_j(z) \rvert^p e^{-\psi(z)} \diff\mu(z) \rightarrow 0,
  \end{align*}
  whenever $\{f_j\}$ is a bounded sequence in $A^p_{\psi}(\Bn)$ which converges to $0$ uniformly on compact subsets.
\end{definition}

\begin{theorem}\label{van_Carl measure char}
  Let $\diff\mu$ be a positive Borel measure. The following statements are equivalent:
  \begin{enumerate}
    \item[(a)] The measure $\diff\mu$ is a vanishing Carleson measure for $A^p_{\psi}(\Bn)$; 
    \item[(b)] 
    ${\widehat \mu}_p (z)\rightarrow 0 \quad \text{as} \quad z \rightarrow \partial \Bn$;
    \item[(c)] For $z \in \Bn$ and sufficiently small $r>0$,
    \begin{align*}
     \frac{\mu(B_{H}(z,r))}{v(B_{H}(z,r))} \rightarrow 0 \quad \text{as} \quad z \rightarrow \partial \Bn \textit{;}
    \end{align*}
    \item[(d)] There is a $H_\psi$-lattice $\{a_k\}$ such that
    \begin{align*}
      \frac{\mu(B_{H}(a_k,r))}{v\!\left(B_{H}(a_k,r)\right)} \rightarrow 0 \quad \text{as} \quad k \rightarrow \infty.
    \end{align*}
  \end{enumerate}
\end{theorem}

\begin{proof}
  Suppose $\mu$ is a vanishing Carleson measure for $A^p_{\psi}(\Bn)$.
  Since $\widetilde{\Phi}_{p,z} \in A^p_{\psi}(\Bn)$ with $\|\widetilde{\Phi}_{p,z}\|_{p,\psi}=1$ and 
$$
\widetilde{\Phi}_{p,z}(w) \simeq \frac{\Phi_{p,z} (w)}{(1-\lvert z \rvert^2)^{\frac{2n+1}{p}}} 
= \frac{e^{\frac{2}{p}\frac{1}{1-\langle w,z\rangle} -\frac{1}{p}\frac{1}{1-\lvert z \rvert^2}}}{(1-\lvert z \rvert^2)^{\frac{2n+1}{p}}}
$$ 
converges to $0$ uniformly on compact subsets of $\Bn$ as $\lvert z \rvert \rightarrow 1^{-}$, we have (b).      

  Relation \eqref{char_Carl_b->c} yields the implication (b) $\Rightarrow$ (c). 

  Suppose the statement (c). For a $H_{\psi}$-lattice, $a_k$ goes to the boundary of $\Bn$ as $k \rightarrow +\infty$ which gives (d).
  
  Let $\{f_j\}$ be a bounded sequence in  $A^p_{\psi}(\Bn)$ which converges to $0$ uniformly on compact subsets, 
and let $$I_j=\int_{\Bn} {\lvert f_j(z)\rvert^p}e^{-\psi(z)}\diff\mu(z).$$ 
  By Lemma \ref{SMVP} and Lemma \ref{covering lem 2} (the same way to Theorem \ref{Carl measure char}), we obtain
  \begin{align*}
    I_j
    \le \sum_{k=1}^{\infty} \frac{\mu(B_H(a_k, r))}{v\!\left(B_{H}(a_k,r)\right)} \int_{B_H(a_k, 2r)} \lvert f_j(w)\rvert^p e^{-\psi(w)} \diff v(w).
  \end{align*} 
  Since $\frac{\mu(B_H (a_k,r))}{v\left(B_{H}(a_k,r)\right)} \rightarrow 0$ ~as~ $k \rightarrow +\infty,$ for any $\varepsilon>0$, there is a positive integer $M$ such that for every $k>M,$ we have $$\frac{\mu(B_H (a_k,r))}{v\!\left(B_{H}(a_k,r)\right)} \leq \varepsilon.$$
  It gives that
    \begin{align*}
    I_j \leq C \sum_{k=1}^M & \int_{B(a_k,2r)} {\lvert f_j(w)\rvert^p}e^{-\psi(w)} \diff v(w)\\
     + & \varepsilon \sum_{k=M+1}^{+\infty} \int_{B(a_k,2r)} {\lvert f_j(w)\rvert^p}e^{-\psi(w)} \diff v(w).
  \end{align*}
  Because the sequence $\{f_j\}$ converges to $0$ uniformly on $\overline{B(a_k,2r)}$, the first summation also converges to $0$ as $j\rightarrow +\infty $. The second summation is dominated by the norm of the function $f_j$ by Lemma \ref{covering lem 2}, namely,
  \begin{align*}
    \sum_{k=M+1}^{+\infty} \int_{B(a_k,2r)} {\lvert f_j(w)\rvert^p}e^{-\psi(w)} \diff v(w) \leq N \|f_j\|_{p,\psi}^p.
  \end{align*}
  Therefore, we have
  \begin{align*}
    \limsup\limits_{j\rightarrow +\infty} I_j 
\leq \varepsilon N \sup_j\|f_j\|_{p,\psi}^p.
  \end{align*}
  Because $\varepsilon$ is arbitrary, the limit of $I_j$ is zero. 
  This completes the proof.
\end{proof}

As we can see from statements (c) and (d) in Theorem \ref{Carl measure char} and Theorem \ref{van_Carl measure char}, the property of (vanishing) Carleson measure does not depend on $p$. 
When the indication of $p$ is not necessary, we will call it a (vanishing) $\psi$-Carleson measure instead of a (vanishing) Carleson measure for $A^p_{\psi}(\Bn)$.

\section{Applications}\label{sec4}
\subsection{Boundedness and compactness of Ces\`{a}ro operators}

Originally, the extended Ces\`{a}ro operator is defined on analytic function spaces on the unit disk:
\begin{align}\label{Vol_1}
  V_g f(z)=\int_0^z f(t) g'(t) \diff t, \quad z\in {\mathbb D}. 
\end{align}
In 1977, Pommerenke \cite{MR454017} defined $V_g$ and studied on the boundedness of the operator on Hardy space $H^2({\mathbb D})$. 
In $n$-dimensional case, Hu \cite{MR1963765} introduced  the extended Ces\`{a}ro operator $V_g$ on the unit ball by means of radial derivative. 
The following is the definition of the operator $V_g$ for $n$-dimensional spaces:
\begin{definition}
  For $g \in \mathcal{O}(\Bn)$,
  \begin{align}\label{Vol_n}
    V_g f(z) := \int_{0}^{1} f(tz)\mathcal{R} g(tz) \frac{\diff t}{t}, \quad z \in \Bn,
  \end{align}
  where $\mathcal{R} g(z) := \sum_{j=1}^{n} z_j \frac{\partial g}{\partial z_j}(z)$.
\end{definition}
\noindent One can see (\ref{Vol_n}) is same as (\ref{Vol_1}) when $n=1$. 

\begin{rmk}
  Let $f$ belong to $\mathcal{O}(\Bn)$. Following \cite{MR3632468}, we can get 
  \begin{align}\label{norm_eq_p}
    \int_{\Bn} \lvert f(z)\rvert^{p} e^{-\psi(z)} \diff v(z) 
\simeq \lvert f(0)\rvert^{p} + \int_{\Bn} \lvert \mathcal{R}f(z)\rvert^p (1-\lvert z \rvert^2)^{2p} e^{-\psi(z)} \diff v(z). 
  \end{align}
  It gives 
  \begin{align*}
    f \in A^p_{\psi}(\Bn) \iff (1-\lvert z \rvert^2)^2\mathcal{R}f(z) \in L^p(\Bn, e^{-{\psi}}\diff v),
  \end{align*}
  which has an essential role of the proof of Theorem \ref{bdd Ces} and Theorem \ref{cpt Ces}.
\end{rmk}

\begin{theorem}\label{bdd Ces}
  Let $g \in \mathcal{O}(\Bn)$. The following statements are equivalent:
  \begin{enumerate}
    \item[(a)] $V_g$ is bounded on $A^p_{\psi}(\Bn)$; 
    \item[(b)] $\lvert  \mathcal{R}g(z)\rvert^p (1-\lvert z \rvert^2)^{2p} \diff v(z)$ is a Carleson measure for $A^p_{\psi}(\Bn)$;
    \item[(c)] $\lvert  \mathcal{R}g(z)\rvert (1-\lvert z \rvert^2)^2$ is bounded. 
  \end{enumerate}
\end{theorem}

\begin{proof}
  By the relation \eqref{norm_eq_p} and the fact $V_g f(0)=0$ and 
$\mathcal{R} V_g f(z)=f(z)\mathcal{R} g(z)$ as in \cite{MR1963765}, we get 
  \begin{align*}
    \|V_g f\|^p_{p, \psi} 
    &\simeq \lvert V_g f(0)\rvert^p + \int_{\Bn} \lvert \mathcal{R}V_g f(z)\rvert^p (1-\lvert z \rvert^2)^{2p} e^{-\psi(z)} \diff v(z)\\
    &=\int_{\Bn} \lvert f(z)\rvert^p \lvert \mathcal{R}g(z) \rvert^p (1-\lvert z \rvert^2)^{2p} e^{-\psi(z)} \diff v(z),
  \end{align*}
  which means 
  \begin{align}\label{carl_1}
     \|f\|_{p,\mu}^p \simeq \|V_g f\|^p_{p, \psi},
  \end{align}
where $\diff\mu(z)=\lvert  \mathcal{R}g(z)\rvert^p (1-\lvert z \rvert^2)^{2p} \diff v(z).$
  It asserts that (a) implies (b). 

  Next, suppose $\diff\mu(z)=\lvert  \mathcal{R}g(z)\rvert^p (1-\lvert z \rvert^2)^{2p} \diff v(z)$ is a Carleson measure for $A^p_{\psi}(\Bn)$. Lemma \ref{SMVP} and \eqref{comp in B_H} yield
  \begin{align*}
    \lvert \mathcal{R}g(z)\rvert^p (1-\lvert z \rvert^2)^{2p} &\le (1-\lvert z \rvert^2)^{2p} \frac{1}{v\!\left(B_H (z, r)\right)} 
\int_{B_H(z, r)} \lvert \mathcal{R}g(w)\rvert^p e^{-\psi(w)} \diff v(w) \\
    &\simeq \frac{1}{v\!\left(B_H (z, r)\right)} \int_{B_H(z, r)} 
\lvert \mathcal{R}g(w)\rvert^p (1-\lvert w \rvert^2)^{2p}e^{-\psi(w)} \diff v(w),
  \end{align*}
  i.e.,
  \begin{align}\label{carl_2}
    \lvert \mathcal{R}g(z)\rvert^p (1-\lvert z \rvert^2)^{2p} \lesssim \frac{\mu(B_H (z, r))}{v(B_H (z, r))}.
  \end{align}
  The last term is dominated by some constant with aid of Theorem \ref{Carl measure char}. It shows (b) implies (c).

Suppose (c) holds, then
\begin{align}
   \begin{split}\label{carl_3}
    \|V_g f\|^p_{p, \psi} 
    &\simeq \int_{\Bn} \lvert f(z)\rvert^p \lvert \mathcal{R}g(z) \rvert^p (1-\lvert z \rvert^2)^{2p} e^{-\psi(z)} \diff v(z)\\
 &\leq\sup_{\Bn}\left\{\lvert \mathcal{R}g(z) \rvert^p (1-\lvert z \rvert^2)^{2p}\right\}
\int _{\Bn} \lvert  f\left( z\right) \rvert ^{p} e^{-\psi(w)}\diff v\!\left( z\right).
\end{split}
  \end{align}
It gives (c) implies (a).

\end{proof}

\begin{theorem}\label{cpt Ces}
  Let $g \in \mathcal{O}(\Bn)$. The following statements are equivalent:
  \begin{enumerate}
    \item[(a)] $V_g$ is compact on $A^p_{\psi}(\Bn)$;
    \item[(b)] $\lvert  \mathcal{R}g(z)\rvert^p (1-\lvert z \rvert^2)^{2p} \diff v(z)$ is a vanishing Carleson measure for $A^p_{\psi}(\Bn)$;
    \item[(c)] $\lvert  \mathcal{R}g(z)\rvert (1-\lvert z \rvert^2)^2 \rightarrow 0$ as $\lvert z \rvert \rightarrow 1^{-}$. 
  \end{enumerate}
\end{theorem}

\begin{proof}
  Similar to the proof of Theorem \ref{bdd Ces}, \eqref{carl_1}, \eqref{carl_2}, and \eqref{carl_3} yield the implications $(a) \Rightarrow (b)$, $(b) \Rightarrow (c)$, and $(c) \Rightarrow (a)$, respectively. 
\end{proof}

\subsection{Boundedness and compactness of Toeplitz operators}

\begin{definition}
  The Toeplitz operator with symbol $u$ is 
  \begin{align*}
    T_{u} f(z) = \int_{\Bn} K(z,w)f(w)u(w)e^{-\psi(w)}\diff v(w)
  \end{align*}
  for $A^2_\psi(\Bn)$.
\end{definition}

Throughout this section, we consider $\diff\mu(z) = u(z) \diff v(z)$ for positive function $u$ and define ${\widehat u}(z) := {\widehat \mu}_2(z)$. Then
\begin{align*}
  {\widehat u} (z) = \int_{\Bn} \lvert\widetilde{\Phi}_{2,z}(w)\rvert^2 e^{-\psi(w)} u(w)\diff v(w)
\end{align*}
and 
\begin{align}\label{imm_berezin}
  {\widehat u} (z) =  \langle T_{u}\widetilde{\Phi}_{2,z}, \widetilde{\Phi}_{2,z} \rangle_{\psi}
\end{align}
by the reproducing property of the Bergman kernel. We can see ${\widehat u}(z)$ plays a similar role to Berezin transform which defined with a normalized Bergman kernel.

\begin{theorem}\label{bdd_Toeplitz}
  Let $u$ be a positive function in $ L^2_\psi(\Bn)$. The following statements are equivalent:
  \begin{enumerate}
    \item[(a)] $T_{u}$ is bounded on $A^2_\psi(\Bn)$;
    \item[(b)] ${\widehat u}$ is a bounded function on $\Bn$;
    \item[(c)] $u \diff v$ is a $\psi$-Carleson measure.
  \end{enumerate}
\end{theorem}

\begin{proof} 
  Suppose that $T_{u}$ is bounded on $A^2_\psi(\Bn)$. We have
  \begin{align}\label{toe_bdd 1}
    \lvert {\widehat u}(z)\rvert = {\lvert \langle T_{u}\widetilde{\Phi}_{2,z}, \widetilde{\Phi}_{2,z} \rangle_{\psi}\rvert} \le {\left\|T_{u}\widetilde{\Phi}_{2,z}\right\|_{2,\psi}}.
  \end{align}
  Since ${\left\|T_{u}\widetilde{\Phi}_{2,z}\right\|_{2,\psi}} \le \left\|T_{u}\right\|$, (a) implies (b).

  Let $\diff\mu(z) = u(z) \diff v(z)$. We will show that (b) implies $\sup \left\{\frac{\mu(B_H (z,r))}{v(B_H (z, r))} ; z \in \Bn \right\} < +\infty$, which is equivalent that $\diff\mu$ is a $\psi$-Carleson measure.
  For any $z \in \Bn$ and a sufficiently small $r>0$,
  \begin{align*}
    \widehat{u} (z) &= \int_{\Bn} \lvert \widetilde{\Phi}_{2,z}(\zeta)\rvert^2 e^{-\psi(\zeta)}\diff\mu(\zeta) \\
    &\ge \int_{B_H(z,r)} \lvert \widetilde{\Phi}_{2,z}(\zeta)\rvert^2 e^{-\psi(\zeta)}\diff\mu(\zeta) \\
    &\simeq \int_{B_H(z,r)} \frac{1}{(1-\lvert z \rvert^2)^{2n+1}} \lvert e^{\frac{1}{1-\langle \zeta,z\rangle}}\rvert^2 e^{-\frac{1}{1-\lvert z \rvert^2}} e^{-\frac{1}{1-\lvert \zeta\rvert^2}}\diff\mu(\zeta)
  \end{align*}
  by \eqref{normalized test function}. It is obtained 
  \begin{align}\label{toe_bdd 2}
    \widehat{u} (z) 
    \gtrsim \frac{1}{(1-\lvert z \rvert^2)^{2n+1}} \int_{B_H(z,r)} \diff\mu(\zeta)
    \simeq \frac{\mu \!\left(B_H(z,r)\right)}{v\!\left(B_H(z,r)\right)}
  \end{align}
  by \eqref{test fn_1} in Lemma \ref{test fn lem} and Corollary \ref{cor_vol}.

 Fubini's theorem and the reproducing property of the Bergman kernel give that 
  \begin{align*}
    \langle T_{u}f, g \rangle_{\psi}
    &= \int_{\Bn} T_{u} f(w) \overline{g(w)} e^{-\psi(w)} \diff v(w)\\
    &= \int_{\Bn} f(\zeta) \overline{g(\zeta)} e^{-\psi(\zeta)}\diff\mu(\zeta).
  \end{align*}
  By H\"{o}lder's inequality, 
  \begin{align}\label{toe_bdd 3}
    \lvert \langle T_{u}f, g \rangle_{\psi}\rvert \le \left\|f\right\|_{2,\mu} \left\|g\right\|_{2,\mu}.
  \end{align}
  Since $\diff \mu$ is a $\psi$-Carleson measure for $A^2_{\psi}(\Bn)$, we have
   \begin{align*}
    \lvert \langle T_{u}f, g \rangle_{\psi}\rvert \le C \left\|f\right\|_{2,\psi} \left\|g\right\|_{2,\psi}.
  \end{align*}
  Therefore, 
  \begin{align*}
    \|T_{u}f\|_{2,\psi} = \sup_{\|g\|_{2, \psi}=1} \lvert \langle T_{u} f, g\rangle_{\psi}\rvert \le C \left\|f\right\|_{2,\psi}.
  \end{align*}  
  It shows (c) implies (a).
\end{proof}

\begin{theorem}\label{cpt Toeplitz}
  Let $u$ be a positive function in $ L^2_\psi(\Bn)$. The following statements are equivalent:
  \begin{enumerate}
    \item[(a)] $T_{u}$ is compact on $A^2_\psi(\Bn)$;
    \item[(b)] $\lvert{\widehat u}(z)\rvert \rightarrow 0$ as $\lvert z \rvert \rightarrow 1^{-}$;
    \item[(c)] $u \diff v$ is a vanishing $\psi$-Carleson measure. 
  \end{enumerate}
\end{theorem}

\begin{proof}
  Inequalities \eqref{toe_bdd 1} and \eqref{toe_bdd 2} assert the implications $(a) \Rightarrow (b)$ and $(b) \Rightarrow (c)$, respectively.

  Suppose $\diff \mu = u \diff v$ is a vanishing $\psi$-Carleson measure. Since $\diff \mu$ is also a $\psi$-Carleson measure, \eqref{toe_bdd 3} implies
  \begin{align*}
    \|T_{u}f\|_{2,\psi} = \sup_{\|g\|_{2, \psi}=1} \lvert \langle T_{u} f, g\rangle_{\psi}\rvert \le C \left\|f\right\|_{2,\mu}.
  \end{align*}  
 Since $\diff \mu$ is a vanishing $\psi$-Carleson measure,
   \begin{align*}
    \|T_{u}f_j\|_{2,\psi} \le C \left\|f_j\right\|_{2,\mu} \rightarrow 0, 
  \end{align*}  
  whenever $\{f_j\}$ is a bounded sequence in $A^p_{\psi}(\Bn)$ which converges to $0$ uniformly on compact subsets.
  It completes the proof.
\end{proof}

When the symbol function $u$ is subharmonic, we get further results on $T_u$.
\begin{corollary}\label{cor_u_subh}
  Let $u$ be a positive function in $ L^2_\psi(\Bn)$. If the symbol $u$ is subharmonic,  
  then the following statements are equivalent:
  \begin{enumerate}
    \item[(a)] $T_{u}$ is bounded on $A^2_\psi(\Bn)$;
    \item[(b)] $u$ is a bounded function on $\Bn$.
  \end{enumerate}
\end{corollary}

\begin{proof}
  If the symbol function $u$ is bounded, then the Toeplitz operator $T_u$ is bounded.

  Since the symbol $u$ is subharmonic, Lemma \ref{SMVP} with $p=1$ gives
  \begin{align*}
    \lvert u(z)\rvert  
    & \lesssim  \frac{1}{v\!\left(B_H (z, r)\right)} \int_{B_H(z, r)} \lvert u(w)\rvert e^{-\psi(w)} \diff v(w)
  \end{align*} 
  for some small $r>0$.
  Boundedness of the operator $T_u$ implies $\frac{\mu(B_{H}(z,r))}{v\left(B_{H}(z,r)\right)}  \le C $ where $\diff \mu = u \diff v$ by Theorem \ref{bdd_Toeplitz}. 
  Hence, $u$ is bounded on $\Bn$.
\end{proof}

\begin{corollary}
  Let $u$ be a positive function in $ L^2_\psi(\Bn)$. If the symbol $u$ is subharmonic,  
  then the following statements are equivalent:
  \begin{enumerate}
    \item[(a)] $T_{u}$ is compact on $A^2_\psi(\Bn)$;
    \item[(b)] $u(z) \rightarrow 0$ when $z \rightarrow \partial \Bn$.
  \end{enumerate}
\end{corollary}

\begin{proof}
  Similar to the proof of Corollary \ref{cor_u_subh}, Lemma \ref{SMVP} and Theorem \ref{cpt Toeplitz} give the result.
\end{proof}





\begin{thebibliography}{10}

\bibitem{MR141789}
{\sc L.~Carleson}, {\em Interpolations by bounded analytic functions and the
  corona problem}, Ann. of Math. (2), 76 (1962), pp.~547--559.

\bibitem{MR3632468}
{\sc H.~R. Cho and I.~Park}, {\em Ces\`aro operators in the {B}ergman spaces
  with exponential weight on the unit ball}, Bull. Korean Math. Soc., 54
  (2017), pp.~705--714.

\bibitem{MR650200}
{\sc J.~A. Cima and W.~R. Wogen}, {\em A {C}arleson measure theorem for the
  {B}ergman space on the ball}, J. Operator Theory, 7 (1982), pp.~157--165.

\bibitem{MR3472830}
{\sc O.~Constantin and J.~A. Pel\'aez}, {\em Integral operators, embedding
  theorems and a {L}ittlewood-{P}aley formula on weighted {F}ock spaces}, J.
  Geom. Anal., 26 (2016), pp.~1109--1154.

\bibitem{MR3406539}
{\sc G.~M. Dall'Ara}, {\em Pointwise estimates of weighted {B}ergman kernels in
  several complex variables}, Adv. Math., 285 (2015), pp.~1706--1740.

\bibitem{MR1656004}
{\sc H.~Delin}, {\em Pointwise estimates for the weighted {B}ergman projection
  kernel in {$C^n$}, using a weighted {$L^2$} estimate for the
  {$\overline\partial$} equation}, Ann. Inst. Fourier (Grenoble), 48 (1998),
  pp.~967--997.

\bibitem{MR2371433}
{\sc M.~R. Dostani\'{c}}, {\em Integration operators on {B}ergman spaces with
  exponential weight}, Rev. Mat. Iberoam., 23 (2007), pp.~421--436.

\bibitem{harville1998matrix}
{\sc D.~A. Harville}, {\em Matrix algebra from a statistician's perspective},
  1998.

\bibitem{MR374886}
{\sc W.~W. Hastings}, {\em A {C}arleson measure theorem for {B}ergman spaces},
  Proc. Amer. Math. Soc., 52 (1975), pp.~237--241.

\bibitem{MR1963765}
{\sc Z.~Hu}, {\em Extended {C}es\`aro operators on mixed norm spaces}, Proc.
  Amer. Math. Soc., 131 (2003), pp.~2171--2179.

\bibitem{MR687635}
{\sc D.~Luecking}, {\em A technique for characterizing {C}arleson measures on
  {B}ergman spaces}, Proc. Amer. Math. Soc., 87 (1983), pp.~656--660.

\bibitem{MR3899961}
{\sc X.~Lv}, {\em Embedding theorems and integration operators on {B}ergman
  spaces with exponential weights}, Ann. Funct. Anal., 10 (2019), pp.~122--134.

\bibitem{MR2538941}
{\sc J.~Marzo and J.~Ortega-Cerd\`a}, {\em Pointwise estimates for the
  {B}ergman kernel of the weighted {F}ock space}, J. Geom. Anal., 19 (2009),
  pp.~890--910.

\bibitem{oleinik1978embedding}
{\sc V.~L. Oleinik}, {\em Embedding theorems for weighted classes of harmonic
  and analytic functions}, J. Sov. Math., 9 (1978), pp.~228--243.

\bibitem{MR1074530}
{\sc V.~L. Oleinik and G.~S. Perelman}, {\em Carleson's embedding theorem for a
  weighted {B}ergman space}, Mat. Zametki, 47 (1990), pp.~74--79, 159.

\bibitem{MR2679024}
{\sc J.~Pau and J.~A. Pel\'{a}ez}, {\em Embedding theorems and integration
  operators on {B}ergman spaces with rapidly decreasing weights}, J. Funct.
  Anal., 259 (2010), pp.~2727--2756.

\bibitem{MR3426615}
{\sc J.~Pau and R.~Zhao}, {\em Carleson measures and {T}oeplitz operators for
  weighted {B}ergman spaces on the unit ball}, Michigan Math. J., 64 (2015),
  pp.~759--796.

\bibitem{MR454017}
{\sc C.~Pommerenke}, {\em Schlichte {F}unktionen und analytische {F}unktionen
  von beschr\"{a}nkter mittlerer {O}szillation}, Comment. Math. Helv., 52
  (1977), pp.~591--602.

\bibitem{MR601594}
{\sc W.~Rudin}, {\em Function theory in the unit ball of {${\bf C}\sp{n}$}},
  vol.~241 of Grundlehren der Mathematischen Wissenschaften [Fundamental
  Principles of Mathematical Science], Springer-Verlag, New York-Berlin, 1980.

\bibitem{MR2891634}
{\sc A.~P. Schuster and D.~Varolin}, {\em Toeplitz operators and {C}arleson
  measures on generalized {B}argmann-{F}ock spaces}, Integr. Equ. Oper. Theory,
  72 (2012), pp.~363--392.

\bibitem{MR3213550}
\leavevmode\vrule height 2pt depth -1.6pt width 23pt, {\em New estimates for
  the minimal {$L^2$} solution of {$\overline\partial$} and applications to
  geometric function theory in weighted {B}ergman spaces}, J. Reine Angew.
  Math., 691 (2014), pp.~173--201.

\bibitem{MR3010276}
{\sc K.~Seip and E.~H. Youssfi}, {\em Hankel operators on {F}ock spaces and
  related {B}ergman kernel estimates}, J. Geom. Anal., 23 (2013), pp.~170--201.

\bibitem{MR35118}
{\sc J.~Sherman and W.~J. Morrison}, {\em Adjustment of an inverse matrix
  corresponding to a change in one element of a given matrix}, Ann. Math.
  Statistics, 21 (1950), pp.~124--127.

\bibitem{MR2115155}
{\sc K.~Zhu}, {\em Spaces of holomorphic functions in the unit ball}, vol.~226
  of Graduate Texts in Mathematics, Springer-Verlag, New York, 2005.

\end{thebibliography}

\end{document}